\documentclass[11pt]{amsart}
\usepackage{geometry} 
\usepackage{pgfplots}
\usepackage{psfrag}
\usepackage{amsmath}
\usepackage{calc}
\usepackage{systeme}
\usepackage{amsfonts}
\usepackage{amsthm}
\usepackage{algorithm2e}
\usepackage[applemac]{inputenc}
\usepackage{bbold}
\usepackage[numbers]{natbib}

\newcommand{\R}{\mathbb{R}}
\newcommand{\be}{\begin{equation}}
\newcommand{\ee}{\end{equation}}
\newcommand{\ba}{\begin{aligned}}
\newcommand{\ea}{\end{aligned}}
\newcommand{\eps}{\epsilon}
\newcommand{\Pt}{\mathcal{P}_2}
\newcommand{\mpr}{\Pt(\R^d)}

\newtheorem{Theorem}{Theorem}[section]
\newtheorem{Lemma}[Theorem]{Lemma}
\newtheorem{Rem}[Theorem]{Remark}
\newtheorem{Def}[Theorem]{Definition}
\newtheorem{Prop}[Theorem]{Proposition}
\newtheorem{Cor}[Theorem]{Corollary}

\newcommand{\triplenorm}[1]{{\left\vert\kern-0.25ex\left\vert\kern-0.25ex\left\vert #1 
    \right\vert\kern-0.25ex\right\vert\kern-0.25ex\right\vert}}

\title{An approximation of the squared Wasserstein distance and an application to Hamilton-Jacobi equations}

\begin{document}





\author{Charles Bertucci \textsuperscript{a},  Pierre-Louis Lions\textsuperscript{b,c}
}
\address{ \textsuperscript{a} CMAP, Ecole Polytechnique, UMR 7641, 91120 Palaiseau, France,\\ \textsuperscript{b} Universit\'e Paris-Dauphine, PSL Research University,UMR 7534, CEREMADE, 75016 Paris, France,\\ \textsuperscript{c}Coll\`ege de France, 3 rue d'Ulm, 75005, Paris, France. }

\maketitle

\begin{abstract}We provide a simple $C^{1,1}$ approximation of the squared Wasserstein distance on $\R^d$ when one of the two measures is fixed. This approximation converges locally uniformly. More importantly, at points where the differential of the squared Wasserstein distance exists, it attracts the differentials of the approximations at nearby points. Our method relies on the Hilbertian lifting of PL Lions and on the regularization in Hilbert spaces of Lasry and Lions. We then provide an application of this result by using it to establish a comparison principle for an Hamilton-Jacobi equation on the set of probability measures.
\end{abstract}



\setcounter{tocdepth}{1}

\tableofcontents
\section{Introduction}
In recent years, the study of partial differential equations on the set of probability measures, sometimes restricted through some integrability condition, has gained a lot of interest. Several difficulties arise when using traditional analysis methods and one of them is that quite often, one would like to use the Wasserstein distance as the natural distance on this set. A main difficulty lies in the fact that, contrary to Euclidean cases, the squared of this function is not smooth. It is only super-differentiable everywhere. We provide here a tool that we believe allows to address this issue in several situations. Namely, we present an approximation of the squared Wasserstein distance which is quite regular. More importantly, this approximation is somehow coherent with the differentiability of the squared Wasserstein distance in the sense that, at points where the differential of the squared Wasserstein distance exists, it attracts the differentials of the approximations at nearby points. 

Our method of proof relies on considering an approximation in the spirit of Lasry and Lions \citep{lasry}. The proof is established by lifting the Wasserstein distance in the space of random variables (the standard lifting introduced by PL Lions in \citep{lionscours}).

We first present our approximation and its convergence. We end this paper by giving an application of this result in the study of an Hamilton-Jacobi equation on the set of probability measures.

\section{Preliminaries}
\subsection{The Wasserstein distance}
We denote by $\Pt(\R^d)$ the space of Borel probability measures on $\R^d$ with a finite second order moment, i.e. probability measures $µ$ such that $M_2(µ) :=\int_{\R^d}|x|^2dµ < \infty$. The set $\Pt(\R^d)$ is a complete metric space endowed with the Wasserstein distance $W_2$ defined by
\[
W_2^2(µ,\nu) = \inf_{\gamma \in \Pi(µ,\nu)}\int_{\R^d\times \R^d} |x-y|^2 \gamma(dx,dy),
\]
where $\Pi(µ,\nu)$ is the set of all couplings between $µ$ and $\nu$, i.e. the set of probability measures on $\R^d\times \R^d$ with first marginal $µ$ and second marginal $\nu$. 

In all that follows, we consider a standard atomeless probabilistic space $(\Omega, \mathcal{A}, \mathbb{P})$ whose expectation is denoted by $\mathbb{E}$. We denote by $H = \mathbb{L}^2(\Omega,\R^d)$, the space of squared integrable random variables on $\R^d$. Recall that $H$ is a separable Hilbert space endowed with the scalar product $\langle X,Y\rangle = \mathbb{E}[X\cdot Y]$ where $x\cdot y$ is the $\R^d$ scalar product between $x$ and $y$. The norm of $X \in H$ is noted  $\|X\|$. We note $X \sim µ$ when $µ$ is the law $\mathcal{L}(X)$ of $X \in H$.

The Wasserstein distance $W_2$ satisfies for all $µ, \nu \in \Pt(\R^d)$,
\[
W^2_2(µ,\nu) = \inf \{\mathbb{E}[|X-Y|^2]| X \sim µ, Y \sim \nu\}.
\]

We have the following immediate result.
\begin{Lemma}\label{lemma:1}
For any $X \sim µ$, we have that
\[
W_2^2(µ,\nu) = \inf \{\mathbb{E}[|X-Y|^2]| Y \sim \nu\}.
\]
\end{Lemma}
\begin{proof}
Consider an optimal couple $(X^*,Y^*)$ in $W_2^2(µ,\nu)$. It is a classical result of probabilities, see for instance Lemma 5.23 in \citep{carmona2017probabilistic}, that for any $\eps > 0$, there exists a measure preserving map $\tau : \Omega \to \Omega$ such that $\|X-X^*\circ \tau\|_\infty \leq \eps$. Observe now that
\[
\mathbb{E}[|X - Y^*|^2] \leq W_2^2(µ,\nu) + 2\eps W_2(µ,\nu) + \eps^2.
\]
Hence the result follows by taking the limit $\eps \to 0$ in the above expression.
\end{proof}

\subsection{Regularity of functions in $\Pt(\R^d)$}
Consider $\Phi : \Pt(\R^d) \to \R$. We say that $\Phi$ is differentiable at $µ \in \Pt(\R^d)$ if there exists $\phi \in L^2_µ(\R^d,\R^d)$ such that for all $µ' \in \mpr$, $\gamma\in \Pi(µ,µ')$,
\[
\Phi(µ') - \Phi(µ) = \int_{\R^d\times \R^d} \phi(x) \cdot (y-x)\gamma(dx,dy) + o\left( \left(\int_{\R^d\times \R^d}|x-y|^2\gamma(dx,dy)\right)^\frac12\right).
\]
In this case, we note for all $x \in \R^d$, $\phi(x) = D_µ\Phi(µ,x)$.
We say that $\Phi$ is continuously differentiable over $\mpr$ if it is differentiable at every point and if for any compact set $K \subset \mpr$, there exists a continuous function $\omega: [0,\infty)\times [0,\infty)$ with $\omega(0) = 0$ such that for any $µ,µ' \in K$, for any coupling $\gamma\in \Pi(µ,µ')$
\[
\int_{\R^d \times \R^d}|D_µ\Phi(µ,x) - D_µ\Phi(µ',y)|^2\gamma(dx,dy) \leq \omega\left(\int_{\R^d\times \R^d}|x-y|^2\gamma(dx,dy)\right).
\]
Finally, we say that $\Phi$ is $C^{1,1}$ if the previous $\omega$ can be taken linear.\\

These notions of regularity might seem unusual stated in this way. Let us emphasize the fact that they are the simple translations of the associate regularity properties for the lift $\tilde \Phi : H \to \R$ defined by $\tilde \Phi (X) = \Phi(\mathcal{L}(X))$. In particular, $\Phi$ is $C^{1,1}$ if and only if $\tilde \Phi$ is $C^{1,1}$, in the usual sense in $H$.\\

Recalling Theorem 3.2 in \citep{alfonsi}, we know that as soon as there exists an optimal (for the quadratic cost) transport map $T: \R^d \to \R^d$ from $µ$ to $\nu$, then $\Phi:µ' \to W_2^2(µ',\nu)$ is differentiable at $µ$ and 
\[
D_µ\Phi(µ,x) =2(x - T(x)).
\]
In particular, at such points, we have the relation
\be\label{eq:1}
\int_{\R^d}|D_µ\Phi(µ,x)|^2µ(dx) = 4 \Phi(µ).
\ee
This relation links the squared of the norm of $D_µ\Phi$ with $\Phi$ itself.

\section{The approximation}
We fix $\nu \in \Pt(\R^d)$. The previous Lemma provides a natural lift of $µ \to W^2_2(µ,\nu)$ to a function of $U:H\to \R$ defined by $W_2^2(\mathcal{L}(X),\nu) =:U(X) = \inf_{Y \sim \nu} \mathbb{E}[|X-Y|^2]$.\\

\subsection{Basic properties}
Note that $U$ is $1$-semi-concave, more precisely, for any $Z \in H$, $X \to U(X) - \|X-Z\|^2$ is concave as the infimum of linear functions since
\[
U(X) - \|X-Z\|^2 = \inf_{Y \sim \nu}\{ \mathbb{E}[(X-Z)\cdot(Z - Y)] + \|Z-Y\|^2\}.
\]
This property hints that a regularization of $U$ by means of a sup-convolution is possible.

For $\delta \in (0,1)$, we consider
\be\label{defPhi}
\Phi_\delta(µ) = \sup_{µ'\in \Pt(\R^d)}\left\{ W_2^2(µ',\nu) - \frac1\delta W_2^2(µ',µ)\right\}.
\ee
Thanks to Lemma \ref{lemma:1}, we can lift $\Phi_\delta$ to $H$ with the formula
\[
\ba
\Phi_\delta(\mathcal{L}(X)) &= \sup_{µ' \in \Pt(\R^d)}\sup_{X'\sim µ'}\left\{ W_2^2(\mathcal{L}(X'),\nu)- \frac1\delta \mathbb{E}[|X-X'|^2]\right\}\\
&=\sup_{X' \in H} \inf_{Y \sim \nu}\left\{ \mathbb{E}[|X'-Y|^2] - \frac1\delta \mathbb{E}[|X-X'|^2]\right\}.
\ea
\]
We shall note $U_\delta(X) = \Phi_\delta(\mathcal{L}(X))$. Note that by definition we always have
\[
W_2^2(µ,\nu) \leq \Phi_\delta(µ).
\]
On the other hand
\[
\Phi_\delta(\mathcal{L}(X)) \leq\inf_{Y \sim \nu}  \sup_{X' \in H} \left\{ \mathbb{E}[|X'-Y|^2] - \frac1\delta \mathbb{E}[|X-X'|^2]\right\}.
\]
For any $Y$, the supremum in $X'$ is reached for $X'$ such that $\delta(X' - Y) = X' - X$, which yields $X' = (1 -\delta)^{-1}(X - \delta Y)$. Using this information, we obtain 
\[
\Phi_\delta(\mathcal{L}(X)) \leq \frac{1}{1-\delta}\inf_{Y \sim \nu}\mathbb{E}[|X-Y|^2].
\]
Hence the convergence of $(\Phi_\delta)_{\delta > 0}$ as $\delta \to 0$ follows from the relation
\be\label{eq:encadrement}
W_2^2(µ,\nu) \leq \Phi_\delta(µ) \leq \frac{1}{1-\delta}W_2^2(µ,\nu).
\ee

\subsection{Further description of the approximation}
We give sufficient conditions under which the second inequality of \eqref{eq:encadrement} is actually an equality. We then present a remark to understand what happens outside the assumptions of the following result.

\begin{Prop}
Assume that $µ,\nu \in \Pt(\R^d)$ are such that the optimal transport maps $T$ and $T'$ from $µ$ toward $\nu$ and vice versa are Lipschitz continuous. Then, 
\[
\delta \leq \max \left\{ \sigma^-, \frac{1}{\sigma^+}\right\} \Rightarrow \Phi_\delta(µ) = \frac{1}{1-\delta}W_2^2(µ,\nu),
\]
where $\sigma^-$ is the essential infimum of the smallest eigenvalue\footnote{Recall that $T$ and $T'$ are monotone maps.} of $DT'$ and $\sigma^+$ is the essential supremum of the largest eigenvalue of $DT$.
\end{Prop}
Before presenting the proof of this result, we insist upon the fact that thanks to the now quite exhaustive literature on Monge-Amp\`ere equations, there are various types of situations in which we can guarantee that the conditions of the previous proposition are satisfied. For instance, it is the case if $µ$ and $\nu$ are absolutely continuous with respect to the Lebesgue measure, that they are supported on respectively $X$ and $Y$, two bounded open sets such that $Y$ is convex, and the densities of $µ$ and $\nu$ are bounded from above and away from $0$ and that they are H\"older continuous on respectively $X$ and $Y$. We refer to the survey \citep{figalli} and the reference therein for more details on the regularity of the transport maps.
\begin{proof}
The proof basically consists in showing that the intermediate measure $µ'$ which achieves the maximum in \eqref{defPhi} is well behaved. Take $(X,Y)$ an optimal coupling for $(µ,\nu)$, in particular, almost surely, $X = T'(Y)$ and $Y = T(X)$. We want to show that the supremum in $U_\delta(X)$ is reached for 
\[
X' = \frac{1}{1-\delta}X - \frac{\delta}{1-\delta} Y = \frac{1}{1-\delta}(T'-\delta Id)(Y).
\]
Remark that if $\delta \leq \sigma^-$, then $T' - \delta Id$ is monotone. On the other hand, since, almost everywhere on the support of $µ$, $T'(T(x)) = x$, it follows that this condition also reduces to $\delta \leq \frac{1}{\sigma^+}$.
Under this smallness assumption on $\delta$, we have in fact more than the monotonicity of the transport map, we have that it is the gradient of some convex function, since $T'$ is. Hence, $(X',Y)$ is an optimal coupling, see for instance Theorem 1.48 in \citep{santambrogio}. This implies that
\[
\ba
\Phi_\delta(µ) &\geq \mathbb{E}[|X'-Y|^2] - \frac{1}{\delta}\mathbb{E}[|X-X'|^2]\\
& = \frac{1}{1-\delta} W_2^2(µ,\nu).
\ea
\]
The result is thus proved.
\end{proof}
We now present a formal remark which we believe helps to understand the case in which $\delta$ is not small enough.
\begin{Rem}
Assume that we are in the case in which $µ$ and $\nu$ are as well behaved as in the previous Proposition and that they have a density which is positive on their support, but that the equality for $\Phi_\delta$ is not reached, which implies thanks to our result that $\delta$ is not small enough. 
Assume that the intermediate measure $µ'$, which achieves the maximum in \eqref{defPhi}, is sufficiently well behaved. Then it must be that for some convex function $g: \R^d\to \R$, $\nabla g_{\#}\nu = µ'$. Since $µ'$ and $\nu$ are sufficiently well behaved, it cannot be that $D^2g = 0$ at some point in the support of $\nu$. But on the other hand, if $D^2 g > 0$ uniformly, then the argument of the previous Proposition shows that the equality much be reached.

We believe this hints strongly at the fact that whenever $\Phi_\delta(µ) < \frac{1}{1-\delta}W_2^2(µ,\nu)$, the intermediate measure $µ'$ has some singularity.
\end{Rem}

\section{Regularity of the approximation}
We now turn to the regularity of $U_\delta$, which will directly imply some regularity for $\Phi_\delta$. Following the computations of Lasry and Lions \citep{lasry}, we obtain semi-concavity and semi-convexity estimates for $U_\delta$.

\begin{Prop}
For any $\delta \in (0,1)$, $U_\delta$ is $C^{1,1}$. More precisely, for any $Z \in H$,
\[
X \to U_\delta(X) + \frac 1\delta \|X-Z\|^2 \text{ is convex,}
\]
\[
X \to U_\delta(X) - \frac{1}{1-\delta} \|X-Z\|^2 \text{ is concave.}
\]
\end{Prop}
\begin{proof}
We simply follow the proof from \citep{lasry}. For any $X \in H$, 
\[
\ba
U_\delta(X) &= \sup_{X' \in H}\inf_{Y \sim \nu}  \mathbb{E}[|X'-Y|^2] - \frac1\delta \mathbb{E}[|X-X'|^2]\\
& = -\frac 1 \delta \|X\|^2 + \sup_{X' \in H}\left\{ \frac2\delta \mathbb{E}[X\cdot X'] - \frac1\delta \|X'\|^2 + \inf_{Y \sim \nu}  \mathbb{E}[|X'-Y|^2] \right\}
\ea
\]
Hence, $U_\delta$ is $\delta^{-1}$ semi-convex.

We want to show that $\Psi$ defined by
\[
\Psi(X,X') = \inf_{Y \sim \nu}\mathbb{E}[|X'-Y|^2] - \frac1\delta \mathbb{E}[|X-X'|^2] - C \|X\|^2
\]
is concave in $(X,X')$ for a well chosen $C > 0$. Clearly
\[
\ba
\Psi(X,X') = (1 - \delta)^{-1}&\|X'\|^2 - \delta^{-1}\|X\|^2 +\inf_{Y \sim \nu}\{ \|Y\|^2 - 2 \mathbb{E}[X'\cdot Y] \}\\
&+2\delta^{-1}\mathbb{E}[X'\cdot X] - C \|X\|^2.
\ea
\]
Since $X' \to \inf_{Y \sim \nu}\{ \|Y\|^2 - 2 \mathbb{E}[X'\cdot Y] \}$ is concave in $X'$, we easily deduce that $\Psi$ is indeed concave if $C \geq (1 - \delta)^{-1}$. Recalling the Lemma in \citep{lasry} for instance, we know that $ X \to \sup_{X'} \Psi(X,X')$ is concave, which implies that $U_\delta$ is $(1 - \delta)^{-1}$ semi-concave. Note that the same computation holds if we replace $C \|X\|^2$ by $C \|X-Z\|^2$.
\end{proof}

Furthermore, the gradient of $U_\delta$ can be computed as we now show.
\begin{Prop}
For any $\delta \in (0,1)$, $X \in H$, there exists a unique $X^* \in H$ such that $U_\delta (X) = U(X^*) - \delta^{-1}\|X-X^*\|^2$. Moreover,
\[
\nabla U_\delta(X) = \frac2\delta (X^* - X).
\]
\end{Prop}
\begin{proof}
The existence of $X^*$ is immediate. Consider any $\tilde X \in H$, we can compute
\[
\ba
U_\delta(\tilde X) - U_\delta (X) &\geq U(X^*) - \delta^{-1}\|X^*-\tilde X\|^2 - U(X^*) + \delta^{-1}\|X^*-X\|^2\\
& = -2 \delta^{-1}(X^*-X)(X- \tilde X) - \delta^{-1}\|\tilde X - X\|^2.
\ea
\] 
Hence it follows that $ \frac2\delta (X^* - X) \in \partial^-U_\delta(X)$, since $U_\delta$ is $C^{1}$, we deduce that for any $X \in H$, $\partial^-U_\delta(X)= \{\nabla U_\delta (X)\}$, from which the result follows.
\end{proof}

\section{Convergence of the gradients}
We now turn to the property that we believe is the most interesting of this approximation: the fact that the gradient of $U_\delta$ converges nicely toward the gradient of $U$, whenever the latter is well defined.
\begin{Theorem}\label{thm:main}
Consider $X_0 \in H$ such that $U$ is differentiable at $X_0$. Consider a sequence $(X_\delta)_{\delta \in (0,1)}$ such that $\|X_\delta - X_0\| \to 0$ as $\delta \to 0$. Then 
\[
\lim_{\delta \to 0} \nabla U_\delta(X_\delta) = \nabla U(X).
\]
\end{Theorem}
\begin{proof}
For any $\delta$, we note $X_\delta^*\in H$ such that $U_\delta(X_\delta) = U(X_\delta^*) - \delta^{-1}\|X_\delta - X_\delta^*\|^2$. We already saw that $2\delta^{-1}(X_\delta^* - X_\delta) = \nabla U_\delta(X_\delta)$. 

Observe that for any $ X \in H$, by definition of $X_\delta^*$,
\[
\ba
U( X) - U(X_\delta^*) &\leq \delta^{-1}\mathbb{E}[|X-X_\delta|^2 - |X_\delta^*-X_\delta|^2]\\
&= 2\delta^{-1}\mathbb{E}[(X-X_\delta^*)(X_\delta^* - X_\delta)] + \delta^{-1}\|X- X_\delta^*\|^2.
\ea
\] 
This implies that $2\delta^{-1}(X_\delta^*- X_\delta) \in \partial^+U(X_\delta^*)$. We now use Moreau's argument of continuity of the super-differential of a concave function. Define $f: X \to U(X) - \|X-X_0\|^2$ and observe that it is concave. Remark that $2\delta^{-1}(X_\delta^*- X_\delta) - 2 (X_\delta^* - X_0) \in \partial^+f(X_\delta^*)$. Since $X_\delta^* \to X_0$ as $\delta \to 0$ and $\partial^+f(X_0) = \{ \nabla U(X_0)\}$, we obtain, using Moreau's result \citep{moreau}, that
\[
\frac2\delta(X^*_\delta - X_\delta) \rightharpoonup_{\delta \to 0} \nabla U(X_0).
\] 
It now remains to show that this convergence is in fact strong. Recall that $2\delta^{-1}(X_\delta^*- X_\delta) \in \partial^+U(X_\delta^*)$. We want to use the Lipschitz regularity of $U$ to show that this implies a useful bound on the norm of $2\delta^{-1}(X_\delta^*- X_\delta)$. For all $X' \in H, \eps > 0$, there exists $Y \in H, Y \sim \nu$, such that $|U(X') - \mathbb{E}[|X' - Y|^2]| \leq \eps$. It then follows that
\[
\ba
 U(X_\delta^*)- U(X') & \leq \mathbb{E}[|X^*_\delta-Y|^2 - |X'- Y|^2] + \eps \\
& \leq \mathbb{E}[|X'- X_\delta^*|^2 + 2(X_\delta^*-X')\cdot(X'-Y)] + \eps.
\ea
\]
This implies in particular that 
\[
\limsup_{X' \to X_\delta^*} \frac{ U(X_\delta^*)- U(X')}{|X_\delta^*-X'|} \leq 2W_2(\mathcal{L}(X_\delta^*),\nu).
\]
Using this estimate in the translation of the fact that $2\delta^{-1}(X_\delta^*- X_\delta) \in \partial^+U(X_\delta^*)$ yields
\[
\|2\delta^{-1}(X^*_\delta - X_\delta)\| \leq 2 W_2(\mathcal{L}(X_\delta^*),\nu).
\]
This implies that
\[
\limsup_{\delta \to 0}\|2\delta^{-1}(X^*_\delta - X_\delta)\| \leq 2 W_2(\mathcal{L}(X_0),\nu) = \|\nabla U(X_0)\|,
\]
where the last equality follows from \eqref{eq:1}. Hence we obtain that the convergence is in fact a strong convergence which proves the claim.\end{proof}
%
%

\begin{Rem}
One of the strength of this result is that the convergence holds whatever the speed of convergence of $(X_\delta)_{\delta \in (0,1)}$ toward $X_0$. Note also that no assumption on $\nu$ is needed here.
\end{Rem}
We now conclude by giving the following Corollary on the regularity of $\Phi: µ \to W_2^2(µ,\nu)$.
\begin{Cor}
Consider $µ \in \mpr$ a point of differentiability of $\Phi$ and sequences $(\delta_n)_{n \geq 0}$ of positive number converging toward $0$ and $(µ_n)_{n \geq 0}$ converging toward $µ$. Then, for any sequence of couplings $(\gamma_n)_{n \geq 0}$ such that $\gamma_n \in \Pi(µ,µ_n)$ and 
\[
\lim_{n \to \infty} \int_{\R^d \times \R^d}|x-y|^2\gamma_n(dx,dy) = 0,
\]
 it holds that
\[
\lim_{n \to \infty}\int_{\R^d\times \R^d}|D_µ\Phi_{\delta_n}(µ_n,y) - D_µ\Phi(µ,x)|^2 \gamma_n(dx,dy)  = 0.
\]
\end{Cor}

\section{Applications to Hamilton-Jacobi equations on the space of probability measures}
\subsection{Setting}
We consider the Hamilton-Jacobi (HJ) equation
\be\label{hjb}
\lambda U + \int_{\R^d}F(µ,x,D_µU(µ,x))µ(dx) - \int_{\R^d} \text{div}_x(D_µU(µ,x))µ(dx) = G(µ) \text{ in } \Pt(\R^d),
\ee
where $\lambda >0, F : \Pt(\R^d)\times \R^d  \times \R^d \to \R$ and $G: \Pt(\R^d) \to \R$ are given and $\text{div}_x$ is the standard divergence operator on $\R^d$. More precise assumptions on $F$ shall be made later on, but we are only concerned here with subquadratic $F$, i.e., there exists $C > 0$ such that
\be\label{eq:Fquad}
\forall µ, x,p, \quad |F(µ,x,p)| \leq \frac{C_F}{2}(1 + |p|^2).
\ee
The equation \eqref{hjb}, or similar counterparts, are namely studied in mean field optimal control problems, see \citep{daudin1,daudin2,cosso,bayraktar,soner,conforti1,conforti2,conforti3,cecchin} for recent developments and the reference therein for more details on this topic, large deviations, see for instance \citep{feng} or \citep{bertucci2024forthcoming} for a slightly different setting or in differential games \citep{cardaliaguet,daudin2} for instance. One of the main difficulty in \eqref{hjb} is that it involves the term with the divergence which is singular. In the absence of such a term, the study of this type of equation is more standard, we refer for instance to \citep{bertucci2023stochastic}.

The standard candidate theory to study this kind of equations is the theory of viscosity solutions. For a presentation of this theory in finite dimensional spaces, we refer to \citep{crandall1992user}. In the space of probability measures, a general theory is lacking at the moment, as different notions are used depending on variations of \eqref{hjb} which are studied. In most cases, the central result needed to develop a proper study of such HJ equations is a comparison principle between a viscosity super-solution and a viscosity sub-solution, which states that the former is larger than the latter.

In finite dimension, viscosity solutions are usually equivalently defined through smooth test functions or through super/sub differentials. Such an equivalence is not obvious in the space of probability measures, see for instance \citep{bertucci2023stochastic}, as they are elements of general super-differentials which are not necessary given by the differential of smooth test functions. Working with sub/super-differentials is somehow similar to allowing the use of the squared Wasserstein distance as a test function.\\

The main result of this section is to prove that, thanks to presence of the term involving the divergence in \eqref{hjb}, we can indeed use the Wasserstein as if it was a smooth test function. The argument is twofold. On one hand we use the approximation introduced above and on the other hand we take advantage of the singular term to prove that point of interest are points at which the Wasserstein distance is differentiable. This strategy was also used in \citep{daudin1}, but here we shall be able to ask for fewer regularity, namely because of the approximation we provided earlier in the paper. Another situation in which the exact same strategy is used is the one studied in \citep{bertucci2024forthcoming}, which is dimension $d = 1$ and where the divergence term is replaced by another term involving the Hilbert transform.

\subsection{Definition of viscosity solutions}
When working with HJ equations involving singular terms, it is natural to perturb the solution by an appropriate functional. This technique has proven quite useful in \citep{daudin1,bertucci2024forthcoming} for instance. Let us insist upon the fact that this idea has been present in the literature on viscosity solutions for quite a long time. In the case of \eqref{hjb}, this function is the entropy defined by
\[
E(µ) := \begin{cases} \int_{\R^d}\log(µ(x))µ(x)dx \text{ if } µ << \text{ } Leb,\\ + \infty \text{ else,}\end{cases}
\]
where $Leb$ is the Lebesgue measure. Note that $E$ is only defined on a dense subset with empty interior of $\Pt(\R^d)$, hence it is difficult to talk about a differential of $E$. Nonetheless, formally it is given by 
\[
"D_µE(µ,x) =\nabla_x \log(µ(x))".
\]
Observe that the size of this quantity can be given by the Fisher information defined by 
\[
\mathcal{I}(µ) = \int_{\R^d}|\nabla_x\log(µ(x))|^2µ(dx).
\]
The link between \eqref{hjb} and $E$ is then made clear through the formal following computation. Given a smooth map $f: \R^d \to \R^d$
\be\label{eq:382}
\ba
\int_{\R^d}\text{div}_x(f(x))µ(dx) &= - \int_{\R^d}f(x) \nabla_xµ(x)dx\\
&= -\int_{\R^d}f(x) \nabla_x(\log(µ(x)))µ(dx).
\ea
\ee
We shall work with the following definition of viscosity solutions.
\begin{Def}
An usc\footnote{usc stands for upper semi-continuous, while lsc stands for lower semi-continuous.} (resp. lsc) function $U: \Pt(\R^d) \to \R$ is a viscosity sub-solution (resp. super-solution) if for any $\kappa \in (0,C_F^{-1})$ (resp. $\kappa \in(-C_F^{-1}, 0)$), $C^{1,1}$ function $\phi : \Pt(\R^d) \to \R$, for any $µ^*$ point of maximum (resp. of minimum) of $µ \to U(µ) - \phi(µ) - \kappa E(µ)$, it holds that $\mathcal{I}(µ^*) < \infty$ and
\[
\ba
\lambda& \phi(µ^*) + \int_{\R^d}F(µ^*,x,D_µ\phi(µ^*,x) + \kappa\nabla_x(\log(µ(x))))µ^*(dx)\\
& - \int_{\R^d} \text{div}_x(D_µ\phi(µ^*,x) +\kappa \nabla_x(\log(µ(x)) ))µ^*(dx) \leq G(µ^*),
\ea
\]
(resp. the opposite inequality holds).
\end{Def}
Note that the previous is well defined by combining $\mathcal{I}(µ^*) < \infty$ and \eqref{eq:Fquad}. In particular, because $\mathcal{I}(µ^*) < \infty$, $µ \to W_2^2(µ,\nu)$ is differentiable at $µ^*$.

\subsection{Using the Wasserstein distance as a test function}
From now on, we assume that $F$ satisfies
\be\label{hypF}
\ba
\forall µ,\nu,x,y,p,q,\quad |F(µ,x,p) - F(µ,x,q)| \leq C_F(1 + M_2(µ) + |x| + |p| + |q|)|p-q|,\\
\forall µ,\nu,x,y,p,q,\quad |F(µ,x,p) - F(\nu,y,p)| \leq C_F(1 + W_2(\barµ,\bar \nu) + |x-y|)(1 +|p|).
\ea
\ee
For such a coupling $F$, we can indeed prove that we can use the squared Wasserstein distance as a test function. This is done in the next result, using slightly more involved function, as it will be helpful later on.
\begin{Prop}\label{prop:visc}
Let $U$ be a bounded viscosity sub-solution of \eqref{hjb} (resp. super-solution). Take $\nu \in \Pt(\R^d)$, $\alpha,\eps > 0, \kappa \in (0,C_F^{-1})$ (resp. $\alpha,\eps < 0$ and $\kappa \in (-C_F^{-1},0)$) and $\bar µ \in \Pt(\R^d)$ a point of strict maximum (resp. of strict minimum) of $µ \to U(µ) - \kappa E(µ) - \frac\alpha2\int_{\R^d}|x|^2µ(dx) - \frac{1}{2\eps}W_2^2(µ,\nu)$. Then, $\mathcal{I}(\barµ) < \infty$ and 
\[
\ba
\lambda &U(\barµ) + \int_{\R^d}F(\barµ,x,U(\barµ),p(x)+ \alpha x )\barµ(dx) - \int_{\R^d} \text{div}_x(p(x) + \alpha x+ \kappa \nabla_x(\log(\barµ(x))))\barµ(dx)\\
& \leq G(\barµ) + C\kappa(\kappa\mathcal{I}(µ_\delta) + \sqrt{\mathcal{I}(\barµ)}(\eps^{-1}W_2(\barµ,\nu) +  M_2(\barµ)),
\ea
\]
where $C$ depends only on $F$ and $p$ is given by $p(x) = \eps^{-1}(x - T(x))$ where $T$ is the optimal transport map of $\bar µ$ toward $\nu$. (Respectively the opposite inequality holds.)
\end{Prop}
\begin{proof}
We only prove the result for sub-solutions, the other case being symmetric. Take $\delta \in (0,1)$ and consider $\Phi_\delta$ defined in \eqref{defPhi}. For any $\delta$, consider $µ_\delta \in \Pt(\R^d)$ a point of maximum of $µ \to U(µ) - \kappa E(µ) - \frac\alpha2\int_{\R^d}|x|^2µ(dx) - \frac{1}{2\eps}\Phi_\delta(µ)$. By assumption, we obtain that
\be\label{eq:425}
\ba
\lambda &U(µ_\delta) + \int_{\R^d}F(µ_\delta,x,\eps^{-1}D_µ\Phi_\delta(µ_\delta,x)+ \alpha x + \kappa\nabla_x(\log(µ_\delta(x))))µ_\delta(dx)\\
& - \int_{\R^d} \text{div}_x(\eps^{-1}D_µ\Phi_\delta(µ_\delta,x) + \alpha x + \kappa \nabla_x(\log(µ_\delta(x))))µ_\delta(dx) \leq G(µ_\delta).
\ea
\ee
Note that $(µ_\delta)_{\delta > 0}$ converges weakly toward $\bar µ$ as $\delta \to 0$. Moreover, by construction of $µ_\delta$, we also obtain that $(M_2(µ_\delta))_{\delta > 0}$ is bounded. Using \eqref{eq:Fquad} and \eqref{eq:382} with $f =\nabla_x (\log(µ))$, as well as standard quadratic estimates, we deduce that 
\[
\ba
(\kappa - C_F\kappa^2) \mathcal{I}(µ_\delta) \leq &\left(\alpha\sqrt{M_2(µ_\delta)} +\eps^{-1}\|D_µ\Phi_\delta(µ_\delta)\|_{L^2(µ_\delta)}\right)\sqrt{\mathcal{I}(µ_\delta)} \\
&+ 2C_F(\alpha^2M_2(µ_\delta) + \eps^{-2}\|D_µ\Phi_\delta(µ_\delta)\|_{L^2(µ_\delta)}^2) + \|G\|_\infty + \lambda\|U\|_\infty.
\ea
\]
 Hence, we obtain that there exists $C>0$ depending only on $\|G\|_\infty, \|U\|_\infty, \lambda$ and $F$ such that
 \[
 \kappa \sqrt{\mathcal{I}(µ_\delta)} \leq C\left(1+ \alpha\sqrt{M_2(µ_\delta)} +\eps^{-1}\|D_µ\Phi_\delta(µ_\delta)\|_{L^2(µ_\delta)}\right).
 \]
 Since $µ \to W_2^2(µ,\nu)$ is differentiable at $\bar µ$, recall that thanks to Theorem \ref{thm:main}, 
\[
\lim_{\delta \to 0}\|D_µ\Phi_\delta(µ_\delta)\|_{L^2(µ_\delta)} = 2W_2(\bar µ,\nu).
\]
 Having these estimates in hand, it now remains to show that we can indeed pass to the limit into \eqref{eq:425}. The upper-semi continuity of $U$ and the continuity of $G$ yields
\[
\lambda U(µ_\delta) \to_{\delta \to 0} \lambda U(\bar µ), \quad \lim_{\delta \to 0}G(µ_\delta) = G(\bar µ).
\]
Observe now that thanks to \eqref{hypF},
\[
\ba
&|F(µ_\delta,x,\eps^{-1}D_µ\Phi_\delta(µ_\delta,x)+ \alpha x + \kappa\nabla_x(\log(µ_\delta(x)))) - F(µ_\delta,x,\eps^{-1}D_µ\Phi_\delta(µ_\delta,x)+ \alpha x )|\\
&\leq C_F\kappa (1 + |\eps^{-1}D_µ\Phi_\delta(µ_\delta,x)| + M_2(µ_\delta)+ \alpha |x| + \kappa|\nabla_x(\log(µ_\delta(x)))|)|\nabla_x(\log(µ_\delta(x)))|.
\ea
\]
Hence, using Cauchy-Schwarz inequality, we deduce that
\be\label{eq:419}
\ba
 \int_{\R^d}&F(µ_\delta,x,\eps^{-1}D_µ\Phi_\delta(µ_\delta,x)+ \alpha x + \kappa\nabla_x(\log(µ_\delta(x))))µ_\delta(dx)\\
 & \geq  \int_{\R^d}F(µ_\delta,x,\eps^{-1}D_µ\Phi_\delta(µ_\delta,x)+ \alpha x)µ_\delta(dx)\\
 & \quad -C_F\kappa\left(\kappa \mathcal{I}(µ_\delta) + \sqrt{\mathcal{I}(µ_\delta)}\left(\eps^{-1}\|D_µ\Phi_\delta(µ_\delta)\|_{L^2(µ_\delta)} + \alpha \sqrt{M_2(µ_\delta)} + M_2(µ_\delta)\right)\right),
\ea
\ee
for a possible different $C_F$, but which depends only on $F$. To pass to the limit in the right hand side, remark that $(\mathcal{I}(µ_\delta))_{\delta}$ and $(M_2(µ_\delta))_{\delta}$ are bounded (in $\delta$). It now remains to show the convergence of the first term of the right hand side of \eqref{eq:419}. In order to do so, consider $X_\delta \in H, X_\delta \sim µ_\delta$ which converges toward $\bar X \in H, \bar X \sim \bar µ$. Let us remark that
\[
\ba
 \int_{\R^d}F(µ_\delta,x,U(µ_\delta),\eps^{-1}D_µ\Phi_\delta(µ_\delta,x)+ \alpha x )µ_\delta(dx)\\
 = \mathbb{E}\left[ F(µ_\delta, X_\delta, U(µ_\delta),\eps^{-1}\nabla U_\delta(X_\delta) + \alpha X_\delta) \right].
 \ea
\]
Using Theorem \ref{thm:main} and the regularity of $F$, we deduce that
\[
\ba
\lim_{\delta \to 0} \int_{\R^d}F(µ_\delta,x,U(µ_\delta),\eps^{-1}D_µ\Phi_\delta(µ_\delta,x)+ \alpha x )µ_\delta(dx)\\
=  \int_{\R^d}F(\bar µ,x,U(\bar µ),\eps^{-1}p(x)+ \alpha x )\barµ(dx).
\ea
\]
To complete the proof, it only remains to remark that the term $\int_{\R^d}\text{div}_x (D_µ\Phi_\delta(µ_\delta,x))µ_\delta(dx)$ converges to the correct limit, which is proved in a separate Lemma.
\end{proof}

\begin{Lemma}\label{lemma:convE}
Under the assumptions and notation of the proof of Proposition \ref{prop:visc},
\[
\ba
\lim_{\delta \to 0}\int_{\R^d}\text{div}_x (D_µ\Phi_\delta(µ_\delta,x))µ_\delta(dx) = \int_{\R^d}\text{div}_x (p(x))\barµ(dx).
\ea
\]
\end{Lemma}
\begin{proof}
Let us start by remarking that
\[
\int_{\R^d}\text{div}_x (D_µ\Phi_\delta(µ_\delta,x))µ_\delta(dx)  = - \int_{\R^d}D_µ\Phi_\delta(µ_\delta,x)\nabla_x(\log(µ_\delta(x)))µ_\delta(dx).
\]
Hence, denoting $\psi_\delta(x) = \nabla_x(\log(µ_\delta(x)))$, we deduce that 
\[
\int_{\R^d}\text{div}_x (D_µ\Phi_\delta(µ_\delta,x))µ_\delta(dx) = -\mathbb{E}[\nabla U_\delta(X_\delta)\psi_\delta(X_\delta)].
\]
Since $(\mathcal{I}(µ_\delta))_\delta$ is bounded, we deduce that $(\psi_\delta(X_\delta))_\delta$ is bounded in $H$, thus it converges weakly to some limit $Z \in H$. Observe that the proof is complete by using a weak-strong convergence result, thanks to Theorem \ref{thm:main}, if we can show that $Z = \nabla_x(\log(\barµ(\bar X)))$. We now show that this is the case, using a form of convexity of $E$. Recalling the results of McCann \citep{mccann}, we know that $E$ is displacement convex. This implies in particular that for any smooth function $f: \R^d \to \R^d$ with bounded derivatives, for $h \in (\frac{-\|D_xf\|_\infty}{2},\frac{-\|D_xf\|_\infty}{2})$ 
\[
h \to E((Id + h f)_\#µ_\delta) \text{ is convex,}
\]
since the optimal transport of $µ_\delta$ toward $(Id + h f)_\#µ_\delta$ is obtained precisely through the map $Id + h f$. This yields
\[
E((Id + h f)_\#µ_\delta) \geq E(µ_\delta) + h\int_{\R^d}\psi_\delta(x)f(x)µ_\delta(dx).
\]
Passing to the limit $\delta \to 0$ produces the required result. Indeed, in the Hilbert space for instance, the previous can be written
\[
E((Id + h f)_\#µ_\delta) \geq E(µ_\delta) + h\mathbb{E}[\psi_\delta(X_\delta)f(X_\delta)],
\]
and passing to the limit $\delta \to 0$, since we can always choose $(X_\delta)_{\delta > 0}$ converging strongly toward $\bar X$, implies
\[
E((Id + h f)_\#\barµ) \geq E(\barµ) + h\mathbb{E}[Zf(\bar X)].
\]
Note that since $(\mathcal{I}(µ_\delta))_\delta$ is bounded, we deduce that $E(µ_\delta) \to E(\bar µ)$. Recall now that $\mathcal{I}(\bar µ) < \infty$ and that the derivative of $E$ at $\bar µ$ in smooth variations can be computed explicitly to obtain that $Z = \nabla_x \log(\bar µ(\bar X))$.
\end{proof}
An immediate corollary of Proposition \ref{prop:visc} is the following.
\begin{Cor}
Under the same notation as in Proposition \ref{prop:visc}, if $\bar µ$ is a point of strict maximum of $µ \to U(µ) - \kappa E(µ) - \frac\alpha2\int_{\R^d}|x|^2µ(dx) - \int_{\R^d}gdµ - \frac{1}{2\eps}W_2^2(µ,\nu)$ for a smooth bounded $g$. Then, $\mathcal{I}(\barµ) < \infty$ and 
\[
\ba
\lambda &U(\barµ) + \int_{\R^d}F(\barµ,x,U(\barµ),p(x)+ \alpha x + \nabla g(x))\barµ(dx)\\
& - \int_{\R^d} \text{div}_x(p(x) + \alpha x+ \nabla g(x) +\kappa \nabla_x(\log(\barµ(x))))\barµ(dx)\\
& \leq G(\barµ) + C\kappa(\kappa\mathcal{I}(µ_\delta) + \sqrt{\mathcal{I}(\barµ)}(\eps^{-1}W_2(\barµ,\nu) +  M_2(\barµ)),
\ea
\]
for the same constant $C$ as in Proposition \ref{prop:visc}. The similar also holds for super-solutions.
\end{Cor}

\subsection{Comparison principle}
We now end this paper by giving a comparison principle for \eqref{hjb}. Note that, unlike most comparison principles found in the literature around \eqref{hjb} \citep{daudin1,daudin2,bayraktar,soner}, the following does not require any Lipschitz continuity on $G$ or on the viscosity sub/super solutions. We are not asking for any additional regularity on the viscosity sub and super solutions and only on uniform continuity on $G$. Let us note that, for simpler Hamiltonians, \citep{daudin2} established a similar results by approximation of measures by combination of Dirac masses.
\begin{Theorem}
Assume that $U$ and $V$ are respectively bounded viscosity sub and super solutions of \eqref{hjb}, with right hand side given by respectively $G^1$ and $G^2$, two mappings from $\Pt(\R^d)$ into $\R$ such that one of them is uniformly continuous. Then, 
\[
\lambda\sup_{\Pt(\R^d)} (U- V) \leq\sup_{\Pt(\R^d)} (G^1 - G^2).
\]
\end{Theorem}
\begin{proof}
Consider $\alpha \in (0,1), \eps, \eta, \kappa  > 0$. Using a variant of Stegall's Lemma \citep{phelps}, such as an immediate extension of Lemma 2.1 in \citep{bertucci2023monotone} for instance, we can always choose $g_1,g_2 \in C^2(\R^d,\R)$ with $\|g_1\|_{C^2}, \|g_2\|_{C^2} \leq \eta$, such that there exists $(\barµ,\bar\nu)$ a point of strict maximum of 
\[
\ba
(µ,\nu) \to U(µ)& - V(\nu) - \frac{1}{2\eps}W_2^2(µ,\nu) - \frac\alpha2(M_2(µ) + M_2(\nu)) - \kappa(E(µ) + E(\nu))\\
 &- \int_{\R^d}g_1(x)µ(dx) - \int_{\R^d}g_2(x)\nu(dx).
 \ea
\]
Using Proposition \ref{prop:visc}, we deduce that $\mathcal{I}(\bar µ) + \mathcal{I}(\bar \nu) < \infty$ and that
\be\label{visc:U}
\ba
\lambda& U(\bar µ) - \int_{\R^d}\text{div}_x(p(x) + \alpha x + \nabla_x g_1(x) + \kappa \nabla_x(\log(\barµ(x))))\barµ(dx) \\
&+ \int_{\R^d}F(\bar µ,x, p(x) + \alpha x+ \nabla_x g_1(x))\barµ(dx)\leq G_1(\bar µ) + C\kappa(\kappa \mathcal{I}(\bar µ) + \sqrt{\mathcal{I}(\barµ)} (\eps^{-1}W_2^2(\barµ,\bar\nu) + M_2(\barµ)),
\ea
\ee
as well as
\be\label{visc:V}
\ba
\lambda& V(\bar \nu)- \int_{\R^d}\text{div}_x(q(x) + \alpha x + \nabla_x g_2(x)- \kappa \nabla_x(\log(\bar \nu(x))))\bar\nu(dx)  \\
&+ \int_{\R^d}F(\bar \nu,x, q(x) + \alpha x+ \nabla_x g_2(x))\bar\nu(dx)\geq G_2(\bar \nu) - C\kappa(\kappa \mathcal{I}(\bar \nu) + \sqrt{\mathcal{I}(\bar\nu)} (\eps^{-1}W_2^2(\barµ,\bar\nu) + M_2(\bar\nu)),
\ea
\ee
where $p(x) = \eps^{-1}(x- T(x))$ and $q(x) = \eps^{-1}(x - T^{-1}(x))$ for $T$ the optimal transport map from $\bar µ$ to $\bar \nu$ for the quadratic cost. Similarly to what we have done in the proof of Proposition \ref{prop:visc}, using \eqref{hypF}, we obtain that
\[
\ba
& \left|\int_{\R^d}F(\bar µ,x, p(x) + \alpha x+ \nabla_x g_1(x))\barµ(dx) -  \int_{\R^d}F(\bar µ,x, p(x) )\barµ(dx)\right|\\
 &\leq C (1 + \eps^{-1}W_2(\barµ,\bar \nu))(\alpha \sqrt{M_2(\bar µ)} + \eta),
 \ea
\]
with a similar inequality for $V$. Combining these informations with taking the difference of \eqref{visc:U} and \eqref{visc:V}, we obtain
\be\label{eq:relvisc}
\ba
\lambda& (U(\bar µ) - V(\bar \nu)) + \int_{\R^d}F(\bar µ,x, p(x) )\barµ(dx) + \kappa(\mathcal{I}(\barµ) + \mathcal{I}(\bar \nu))\\
&\quad- \int_{\R^d}F(\bar \nu,x, q(x))\bar\nu(dx)- \int_{\R^d}\text{div}_x(p(x))\bar µ(dx) +\int_{\R^d}\text{div}_x(q(x))\bar \nu(dx)\\
&\leq G_1(\bar µ) - G_2(\bar \nu) +  C(1 +\eps^{-1}W_2(\barµ,\bar\nu))(\alpha (\sqrt{M_2(\bar µ)} + \sqrt{M_2(\bar\nu)})+ \eta)\\
 &\quad+ C\kappa(\kappa (\mathcal{I}(\bar µ) + \mathcal{I}(\bar \nu)) + (\sqrt{\mathcal{I}(\barµ)} + \sqrt{\mathcal{I}(\bar \nu)}) (\eps^{-1}W_2^2(\barµ,\bar\nu) + M_2(\barµ) + M_2(\bar \nu)).
\ea
\ee
We can estimate
\[
\ba
\int_{\R^d}F(\bar \nu,x, q(x))\bar\nu(dx)-\int_{\R^d}F(\bar \mu,x,p(x) )\barµ(dx)  = &\int_{\R^d}F(\bar \nu,T(x),p(x) ) - F(\bar \mu,x,p(x) )\barµ(dx) \\
\leq & C\int_{\R^d}( 1 + |p(x)|)(W_2(\barµ,\bar\nu) + |T(x) - x|)\barµ(dx)\\
\leq & CW_2(\barµ,\bar\nu) (1 + \eps^{-1}W_2(\barµ,\bar \nu)),
\ea
\]
where we used $T_\#µ= \nu$ in the equality, \eqref{hypF} for the first inequality and Cauchy-Schwarz inequality for the second inequality.\\

Let us now note $\pi$ the optimal coupling between $\barµ$ and $\bar \nu$. We can compute
\[
\ba
I :&= \int_{\R^d}\text{div}_x(p(x))\bar µ(dx) -\int_{\R^d}\text{div}_x(q(x))\bar \nu(dx)\\
&=-\int_{\R^{2d}}(x-y)\cdot(\nabla \log(\barµ(x)) - \nabla \log(\bar \nu(y)))\pi(dx,dy). 
\ea
\]
Using, as in Lemma \ref{lemma:convE} the convexity of $E$, we deduce that
\[
\ba
E(\bar µ) + \int_{\R^{2d}}(y-x)\cdot\nabla\log(\barµ(x))\pi(dx,dy) \leq E(\bar \nu)\\
E(\bar \nu) + \int_{\R^{2d}}(x-y)\cdot\nabla\log(\bar\nu(y))\pi(dx,dy) \leq E(\bar \mu).
\ea
\]
Hence, we obtain that $I \leq 0$. Coming back to \eqref{eq:relvisc}, we deduce that
\be\label{eq:last}
\ba
\lambda (U(\bar µ) - V(\bar \nu)) + \kappa(\mathcal{I}(\barµ) + \mathcal{I}(\bar \nu)) \leq G_1(\bar µ) - G_2(\bar \nu) +C(1 +\eps^{-1}W_2(\barµ,\bar\nu))(\alpha \sqrt{M_2(\bar µ)} + \eta)\\
 + C\kappa(\kappa (\mathcal{I}(\bar µ) + \mathcal{I}(\bar \nu)) + (\sqrt{\mathcal{I}(\barµ)} + \sqrt{\mathcal{I}(\bar \nu)}) (\eps^{-1}W_2^2(\barµ,\bar\nu) + M_2(\barµ) + M_2(\bar \nu))\\
 +  CW_2(\barµ,\bar\nu) (1 + \eps^{-1}W_2(\barµ,\bar \nu)).
\ea
\ee
Using the continuity of either $G_1$ or $G_2$, we deduce that for a modulus of continuity $\omega$, $G_1(\bar µ) - G_2(\bar \nu) \leq \sup(G_1 - G_2) + \omega(W_2(\barµ,\bar \nu))$. The result shall follow by taking all the parameters to $0$, in some precise order, which we now explain.\\

Recall first that from standard estimates, there exists a modulus of continuity $\omega$ such that $\eps^{-1}W_2^2(\bar µ,\bar \nu) \leq \omega(\eps)$ and $\alpha (M_2(\bar\mu) +  M_2(\bar \nu)) \leq \omega(\alpha)$. Moreover, we clearly have $\lambda\sup(U-V) \leq \liminf \lambda U(\bar µ) - V(\bar \nu)$. Concerning the terms in $\kappa$ we make the following argument.

From the proof of Proposition \ref{prop:visc} for instance, $(u_\kappa)_{\kappa > 0} :=(\kappa(\sqrt{\mathcal{I}(\bar µ)} + \sqrt{\mathcal{I}(\bar \nu)}))_{\kappa >0}$ is bounded, with a bound which depends on all the other parameters. Hence, we deduce from \eqref{eq:last} that $(v_\kappa)_{\kappa > 0} :=(\kappa(\mathcal{I}(\bar µ) + \mathcal{I}(\bar \nu)))_{\kappa >0}$ is also bounded. It then implies that $\lim_{\kappa \to 0} u_\kappa = 0$, uniformly in the other arguments, i.e., up to changing $\omega$, we can assume that $u_\kappa \leq \omega(\kappa)$. Hence, combining all the arguments above leads to
\[
\ba
\lambda \sup(U-V) \leq  \sup(G_1 - G_2)& + \omega(\eps) + C(1 + \eps^{-\frac 12}\omega(\eps))(\omega(\alpha) + \eta)\\
& + C\omega(\kappa)(\omega(\eps) + \alpha^{-1}\omega(\alpha) + \omega(\eps)(1 + \eps^{-\frac 12}\omega(\eps))) + C \omega(\eps).
\ea
\] 
Hence the result follows by taking the limit, in this order, $\kappa \to 0, \alpha \to 0, \eta \to 0$ and then $ \eps \to 0$.
\end{proof}

\section*{Acknowledgments}
The authors have been partially supported by the Chair FDD/FIME (Institut Louis Bachelier). 
\bibliographystyle{plainnat}
\bibliography{bibmatrix}

\end{document}